\documentclass[orivec]{llncs}
\usepackage[english]{babel}
\usepackage[T1]{fontenc}
\usepackage[utf8]{inputenc}
\usepackage{cite}
\usepackage{amsmath}
\usepackage{mathtools}
\usepackage{amssymb}
\usepackage{graphicx}
\usepackage{float}
\usepackage{enumerate}
\usepackage{txfonts}
\begin{document}

\title{ Topology of Asymptotic Cones and $\mathcal{S}$-machines }

\author{Anthony Gasperin}

\institute{Theoretical Computer Science Department,University of Geneva, Switzerland\\
	\email{anthony.gasperin@unige.ch}}

\date{\today}

\maketitle

\newtheorem{Def}{Definition}
\newtheorem{Lem}{Lemma}
\newtheorem{Th}{Theorem}
\newtheorem{Conj}{Conjecture}
\newtheorem{Prop}{Proposition}
\newtheorem{Cor}{Corollary}
\newtheorem{Claim}{Claim}
\newtheorem{Rem}{Remark}
\newtheorem{Stat}{Statement}
\begin{abstract}
Sapir, Birget and Rips showed how to construct groups from Turing machines. To achieve such a construction they introduced the  notion of $\mathcal{S}$-machine.
Then considering a simplified $\mathcal{S}$-machine Sapir and Olshanskii showed how to construct a group such that each of its asymptotic cone is
non-simply connected. Still using the notion $\mathcal{S}$-machine, they constructed a group with two asymptotic cone non-homeomorphic. 
In this paper we show that each asymptotic cone of a group constructed following the whole method of Sapir, Birget and Rips is not simply connected.
\end{abstract}

\section{Introduction}

Let $(X,d_X)$ a metric space $s=(s_n)$ a sequence of points in $X$, $d=(d_n)$ an increasing sequences of numbers with 
$\text{lim }d_n=\infty$ and let $\omega: P(\mathbb{N}) \to \{0,1\}$ be a non-principal ultrafilter. An \textit{asymptotic cone} of
$Con_\omega(X,s,d)$ of $(X,d_X)$ is the subset of the cartesian product $X^\mathbb{N}$ consisting of sequences $(x_i)_{i \in \mathbb{N}}$
with $\text{lim}_\omega \frac{d_X( s_i,x_i)}{d_i} < \infty$ where two sequences $(x_i)$ and $(y_i)$ are equivalent if and only if
$\text{lim}_\omega \frac{d_x(x_i,y_i)}{d_i}=0$. The distance between two elements $(x_i),(y_i)$ in the asymptotic cone 
$Con_\omega(X,s,d)$ is defined as $\text{lim}_\omega \frac{d_X(x_i,y_i)}{d_i}$. Here $\text{lim}_\omega$ is defined as follows. If $a_n$
is a bounded sequence of real numbers then $\text{lim}_\omega(a_n)$ is the unique number $a$ such that for every $\epsilon >0$,
$\omega(\ \{ n\ |\ |a_n-a| < \epsilon \}\ )=1$.
The \emph{asymptotic cones} of a finitely generated group $G$ are asymptotic cones of the Cayley graph of $G$ and it well known that 
they do not depend on the choice of the sequence $s$. It is then assumed that $s=(1)$ where $1$ is the identity. 
Given an ultrafilter $\omega$ and an increasing sequence of numbers $d$ the asymptotic cone of a finitely generated group $G$ is then noted $Con_\omega(X,d)$.

A function $f:\mathbb{N} \to \mathbb{N}$ is an isoperimetric function of a finite presentation $\langle X,R \rangle$ of a group $G$ if every word $w$ in $X$,
which is equal to $1$ in $G$, is freely equal to a product of conjugates $\prod^{m}_{i=1} x_i^{-1} r_i x_i$ where $r_i$ or $r^{-1}_i$ is in $R$, $x_i$ is
in $(X \cup X^{-1})^*$ and $m \leq f(|w|)$.
The \emph{Dehn function} of a finite presentation $\langle X,R \rangle$ is defined as the smallest \emph{isoperimetric} function of the presentation.\\
Let $f,g:\mathbb{N}\to \mathbb{N}$ be two functions, $f \preceq g$ if there exists a positive constant $c$ such that $f(n) \leq cg(cn) + cn + c$.
If all functions considered grow at least as fast as $n$ $f(n) \preceq g(n)$ if and only if $f(n) \leq cg(cn)$ for some positive constant $c$.
The results of  Sapir, Birget, Rips and Olshanskii consider only the functions which grow at least as fast as $n$. Two functions $f,g$ are called \emph{equivalent} if $f\preceq g$ and $g\preceq f$.

Since the results of \cite{Alo90,Ger92,Mad85} it is well known that the Dehn function corresponding to different finite presentations of the same group are equivalent. 
Thus one can speak about \emph{the Dehn function} of a finitely presented group. In \cite{Ger92,Mad85} it is shown that the Dehn function of a finitely presented group has a recursive upper bound if and only if the group has a decidable word problem.

In \cite{Bir02,Gro93,Ols91} the connections between Dehn functions, asymptotic geometry of groups and computational complexity of the word problem are discussed.
In \cite{Gro93} Gromov showed that if all \emph{asymptotic cones} of a group $G$ are simply connected then $G$ is finitely presented, has polynomial isoperimetric
function and linear isodiametric function. Papasoglu \cite{Pap96} proved that if a finitely presented group has quadratic isoperimetric function then all its
asymptotic cones are simply connected. 
Sapir, Birget and Rips in \cite{Sap02} introduced the concept of $\mathcal{S}$-machines to show that the word problem of a finitely generated group is decidable
in polynomial time if and only if this group can be embedded into a group with polynomial isoperimetric function.
Olshanskii and Sapir in \cite{Sap05} constructed a group with polynomial isoperimetric function, linear 
isodiametric function and non-simply connected asymptotic cones, the group is roughly a $\mathcal{S}$-machine introduced in \cite{Sap02}. In \cite{Ols05} 
they also constructed a group with two non-homeomorphic asymptotic cones  using the concept of $\mathcal{S}$-machine.

In this paper we show that the whole machinery of Birget, Sapir and Rips leads groups with non-simply connected cones for every Turing machine considered. 
Indeed we show that the construction of \cite{Sap02} involves relations that totally break the topology of the 
asymptotic cones.

\section{Preliminaries}  
This section introduces briefly the machinery introduced by Sapir, Birget and Rips in \cite{Sap02}.
We need to explain, at least superficially, what is a $\mathcal{S}$-machine, how it works and especially how it leads to the construction of groups.

\subsection{$\mathcal{S}$-machines}

This section is closely modeled on \cite{Sap02}, we recall the notion of $\mathcal{S}$-machine defined in the work of Sapir, Birget and Rips in \cite{Sap02}.
To begin let us present the initial assumptions needed for the construction.   
In \cite{Sap02} every Turing machine is modified according to the following lemma:
\begin{Lem}\cite{Sap02} \label{machine} For every Turing machine $M$ recognizing a language $L$ there exists 
a Turing machine $M'$ with the following properties.
\begin{itemize}
	\item The language recognized by $M'$ is $L$.
	\item $M'$ is symmetric, that is, with every command $U \to V$ it contains the inverse command
	$V \to U$.
	\item The time, generalized time, space and generalized space functions of $M'$ are equivalent
	to the time function of $M$. The area function of $M'$ is equivalent to the square of the time function of $M$.
	\item The machine accepts only when all tapes are empty.
	\item Every command of $M'$ or its inverse has one of the following forms for some $i$
	\begin{enumerate}
		\item $\{q_1\omega \to q'_1\omega,\dots,q_{i-1}\omega \to q'_{i-1}\omega,aq_i\omega \to q'_i \omega,q_{i+1}\omega \to q'_{i+1} \omega, \dots \}$
		\item $\{q_1 \omega \to q'_1 \omega,\dots,q_{i-1} \omega \to q'_{i-1}\omega, \alpha q_i \omega \to
		\alpha q'_i\omega,q_{i+1}\omega \to q'_{i+1}\omega,\dots \}$ where "$a$" belongs to the tape alphabet of tape $i$, and $q_j,q'_j$ are state letters of tape $j$.
		\item The letters used on different tapes are from disjoint alphabets. This includes the 
		state letters.
	\end{enumerate}
\end{itemize}
\end{Lem}

Let us present how Sapir, Birget and Rips define a $\mathcal{S}$-machine in \cite{Sap02}, roughly speaking it is defined as a rewriting system. 
A $\mathcal{S}$-machine then comes with a \textit{hardware}, a \textit{language of admissible words}, and
a set of \textit{rewriting rules}.
A \textit{hardware } of a $\mathcal{S}$-machine is a pair $(Y,Q)$ where $Y$ is an $n$-vector of (not necessarily disjoint) sets $Y_i$, $Q$ is an $(n+1)$-vector of 
disjoints sets $Q_i$ with $(\bigcup Y_i) \cap (\bigcup Q_i) = \emptyset$.
The elements of $\bigcup Y_i$ are called \textit{tape letters}; the elements of $\bigcup Q_i$ are called \textit{state letters}.
With every hardware $\mathcal{S}=(Y,Q)$ one can associate the \textit{language of admissible words} $L(\mathcal{S})=Q_1 F(Y_1) Q_2 \cdots F(Y_n)Q_{n+1}$ where
$F(Y_i)$ is the language of all reduced group words in the alphabet $Y_j \cup Y^{-1}_j$.
This language completely determines the hardware. One can then describe the language of admissible words instead of describing the hardware $\mathcal{S}$.
If $1 \leq i < j \leq n$ and $W=q_1 u_1 q_2 \cdots u_n q_{n+1}$ is an admissible word, $q_i \in Q_i, u_i \in (Y_i \cup Y^{-1}_i)^*$ then the 
subword $q_i u_i \cdots q_j$ of $W$ is called the $(Q_i,Q_j)$-subword of $W$ ($i < j)$.
The rewriting rules ( $S$-\textit{rules}) have the following form:
\begin{center}
	$[U_1 \to V_1, \dots, U_m \to V_m ]$
\end{center}
where the following conditions hold:
each $U_i$ is a subword of an admissible word starting with a $Q_l$-letter and ending with $Q_r$-letter.
If $i < j$ then $r(i) < l(j)$, where $r(i)$ is the end of $U_i$ and $l(j)$ the start of $U_j$.
Each $V_i$ is a subword of an admissible word whose $Q$-letters belong to $Q_{l(i)} \cup \cdots \cup Q_{r(i)}$.
The machine applies a $S$-rule to a word $W$ replacing simultaneously subword $U_i$ by subword $V_i,i=1,\dots,m$.

As mentioned in \cite{Sap02} there exists a natural way to convert a Turing machine $M$ into a $\mathcal{S}$-machine $\mathcal{S}$; 
one can concatenate all tapes of the given machine $M$ together and replace every command $aq\omega \to q' \omega$ by $a^{-1}q'\omega$. 
Unfortunately the $\mathcal{S}$-machine constructed following this
natural way will not inherit most of the properties of the original machine $M$, that is it will not satisfy anymore the properties of \emph{Lemma} \ref{machine}.
According to \cite{Sap02} the main problem is that it is nontrivial to construct a $\mathcal{S}$-machine which recognizes only positive powers 
of a letter.
Thus in order to construct a $\mathcal{S}$-machine $\mathcal{S}(M)$ which will inherit the desired properties of a Turing machine $M$, Sapir, Birget and Rips 
in \cite{Sap02} constructed eleven $\mathcal{S}$-machines and then used them to construct the final $\mathcal{S}$-machine $\mathcal{S}(M)$ simulating $M$. 
The construction is quite involved and nontrivial, one can see \cite{Sap02} for details.

Taking any Turing machine $M=\langle X,Y,Q,\Theta,\vec{s}_1,\vec{s}_0 \rangle$ and modifying it according to Lemma \ref{machine}, \cite{Sap02}
constructs a $\mathcal{S}$-machine $\mathcal{S}(M)$ simulating $M$. The $\mathcal{S}$-machine constructed in \cite{Sap02} is quite long to define, next we explain 
briefly the main part of the construction, for proofs and deeper understanding of the whole machinery the reader can refer to \cite{Sap02}. 
The main idea of the construction is to simulate the initial machine $M$ using eleven $\mathcal{S}$-machines $S_1,S_2, \dots,S_9,S_{\alpha},S_{\omega}$.
We will explain how the machines  $S_4,S_9,S_{\alpha},S_{\omega}$ are used in the construction of $\mathcal{S}(M)$. The others $\mathcal{S}$-machines are used to construct 
$S_4$ and $S_9$ and are rather of technical importance.
First we need to describe what is an admissible word of the $\mathcal{S}$-machine $\mathcal{S}(M)$.
For every $q \in Q$ the word $q\omega$ is denoted by $F_q$, in every command of $M$ the word $q\omega$ is replaced by $F_q$.
Left marker on tape $i$ is denoted by $E_i$. This gives a Turing Machine $M'$ such that the configurations of each tape have the form 
$E_i u F_q$ where $u$ is a word in the alphabet of tape $i$ and every command or its inverse has one of the forms:
\begin{equation}\label{positive_rules}
	\{ F_{q_1} \to F_{q'_1}, \dots, aF_{q_i} \to F_{q'_i}, \dots,F_{q_k}\to F_{q'_k} \}
\end{equation}
where $a \in Y$ or 

\begin{equation}
	\{ F_{q_1} \to F_{q'_1}, \dots, E_i F_{q_i} \to E_i F_{q'_i} ,\dots , F_{q_k} \to F_{q'_k} \}.
\end{equation}

An admissible word of the considered $\mathcal{S}(M)$ machine is a product of three parts. 
The first part has the form 
\begin{center}
	$E(0) \alpha^{n_1} x(0) \alpha^{n_2} F(0)$.
\end{center}
The second part is a product of $k$ words of the form
\begin{center}
	$E(i)v_ix(i)w_iF(i)E'(i) p(i) \Delta_{i,1} q(i) \Delta_{i,2}r(i) \Delta_{i,3} s(i) \Delta_{i,4} t(i) \Delta_{i,5}$\\ 
	$u(i) \Delta_{i,6} \overline{p}(i) \Delta_{i,7} \overline{q}(i) \Delta_{i,8} \overline{r}(i) \Delta_{i,9}$\\
	$\overline{s}(i) \Delta_{i,10} \overline{t}(i)\Delta_{i,11} \overline{u}(i) \Delta_{i,12} F'(i), i =1,\dots,k$
\end{center}
The third part has the form 
\begin{center}
	$E'(k+1) \omega^{n'_1}x'(k+1)\omega^{n'_2}F'(k+1)$.
\end{center}
Here $v_i,w_i$ are group words in the alphabet $Y_i$ of tape $i$, and $\Delta_{i,j}$ is a power of $\delta$.
The letters $E(i),x(i),F(i),E'(i),p(i),q(i),r(i),s(i),t(i),u(i),\overline{p}(i),\overline{q}(i),\overline{r}(i)$,
$\overline{s}(i),\overline{t}(i),\overline{u}(i),F'(i)$ belong to disjoint sets of state letters.

The letters $x(i),p(i),q(i),r(i),s(i),t(i),u(i),\overline{p}(i),\overline{q}(i),\overline{r}(i)$,
$\overline{s}(i), \overline{t}(i), \overline{u}(i)$
are called standard and are included into the corresponding sets 
$\mathbf{X}(i),\mathbf{P}(i),\mathbf{R}(i),\mathbf{S}(i),\mathbf{T}(i),\mathbf{U}(i),\overline{\mathbf{P}}(i),\overline{\mathbf{Q}}(i)$
,$\overline{\mathbf{R}}(i),\overline{\mathbf{S}}(i),\overline{\mathbf{T}}(i),\overline{\mathbf{U}}(i),(i=1,\dots,k)$.
Let $\tau$ be a command in $\Theta$ of the form (\ref{positive_rules}) ( command of the form (\ref{positive_rules}) are called \textit{positive } and their inverse \textit{negative}).
For every $\gamma \in \{4,9,\alpha,\omega\}$ and for each component $V(i)$ of the vector of sets of state letters, the letters 
$V(i,\tau,\gamma)$ are included into $V(i)$ where $V \in \{ P, Q,R,S,T,U,\overline{P},\overline{Q},\overline{R},\overline{S},\overline{T},\overline{U}\}$.
For each $\mathcal{S}$-machine $S_{\gamma}, \gamma \in\{4,9,\alpha,\omega\}$ a copy of $S_{\gamma}$ is considered where every state letter $z$ is replaced
by $z(j,\tau,\gamma)$ where $j=i$ if $\gamma=4,9, j=0$ if $\gamma=\alpha$ and $j=k+1$ if $\gamma=\omega$. 
These state letters are included into the corresponding sets. The state letters we just described are all the state letters of $\mathcal{S}(M)$.
The rules of $\mathcal{S}(M)$ are the rules of $\mathcal{S}_4(\tau),\mathcal{S}_9(\tau),\mathcal{S}_{\gamma}(\tau),\mathcal{S}_{\omega}(\tau)$
for all $\tau \in \Theta$ of the form (\ref{positive_rules}) plus the connecting rules.
Basically the connecting rules allow to go from a machine to another one, there are five such rules:
$R_4(\tau),R_{4,\alpha}(\tau),R_{\alpha,\omega}(\tau),R_{\omega,9}(\tau),R_9(\tau)$.
They can be described informally as follows.
$R_4(\tau)$ turns on the machine $\mathcal{S}_4(\tau)$. $R_{4,\alpha}(\tau)$ turns on the machine $\mathcal{S}_{\alpha}(\tau)$ when 
$\mathcal{S}_4(\tau)$ finishes its work, $R_{\alpha,\omega}(\tau),R_{\omega,9}(\tau)$ do the same with the corresponding $\mathcal{S}$-machines.
$R_9(\tau)$ turns off $\mathcal{S}_9(\tau)$ and gets the machine ready to simulate the next transition from $\Theta$.
This machinery contains all the necessary steps to simulate a rule of the machine $M$.

Formally speaking, to every configuration $c=(E_1 v_1F_{q_1},\dots,E_k v_k F_{q_k})$ of the machine $M$ is associated the following
admissible word $\sigma(c)$ of $\mathcal{S}(M)$:\\
$E(0) \alpha^n x(0) F(0)$\\
$E(1)v_1 x(1) F_{q_1}(1)E'(1)p(1)\delta^{||v_1||}q(1)r(1)s(1)t(1)u(1)$\\
$\overline{p}(1) \overline{q}(1) \overline{r}(1) \overline{s}(1)\overline{t}(1)\overline{u}(1)F'_{q_1}(1)\dots$\\
$E(k)v_k x(k) F_{q_k}(k)E'(k)p(k)\delta^{||v_k||}q(k)r(k)s(k)t(k)u(k) $\\
$\overline{p}(k) \overline{q}(k) \overline{r}(k) \overline{s}(k)\overline{t}(k)\overline{u}(k)F'_{q_k}(k)$\\
$E'(k+1)x'(k+1)\omega^n F'(k+1)$, where $||v||$ is the algebraic sum of the degree of the letters in $v$.

The construction of such $\mathcal{S}$-machine allows to construct a group presentation, once again
this part is strongly modeled on \cite{Sap02}.
Let $\mathcal{S}(M)$ be the $S$-machine as constructed before. Let $Y$ be the vector of sets of tape letters, and let $Q$ be the vector
of state letters of $\mathcal{S}(M)$. One can remark that $Q$ has $17k+6$ components which \cite{Sap02} denotes by $Q_1,\dots,Q_{17k+6}$. 
In \cite{Sap02} Sapir, Birget and Rips noticed that $Q_1=\textbf{E}(0),Q_2=\textbf{X}(0),Q_3=\textbf{F}(0), Q_{17k+4}=\textbf{E}'(k+1), Q_{17k+5}=\textbf{X}(k+1),Q_{17k+6}=\textbf{F}'(k+1)$.
Let $\mathbf\Theta_+$ the set of positive rules of $\mathcal{S}(M)$ and $N$ a positive integer. 
To construct their group $G_N(\mathcal{S})$ Sapir, Birget and Rips take the following generating set :
\begin{equation}
	A=\bigcup\limits^{17k+6}_{i=1} Q_i \cup \{ \alpha,\omega,\delta \} \cup \bigcup\limits_{i=1}^{k} Y_i \cup \{ \kappa_j | j=1,\dots,2N \} \cup \mathbf\Theta_+.
\end{equation}
and the following set $P_N(\mathcal{S})$ of relations:
\begin{enumerate}
	\item \textit{Transitions relations}. These relations correspond to elements of $\mathbf\Theta_+$.
	Let $\tau \in \mathbf\Theta_+, \tau=[U_1 \to V_1,\dots,U_p \to V_p]$. Then relations $\tau^{-1}U_1\tau=V_1,\dots,\tau^{-1}U_p\tau=V_p$ are included 
	into $P_N(\mathcal{S})$. If for some $j$ from $1$ to $17k+6$ the letters from $Q_j$ do not appear in any of the $U_i$ then the relations
	$\tau^{-1} q_j \tau =q_j$ for every $q_j \in Q_j$ are also included.

	\item \textit{Auxiliary relations}. These are all possible relations of the form $\tau x= x \tau$ where 
	$x \in \{\alpha,\omega,\delta \} \cup \bigcup_{i=1}^k Y_i,\tau \in \mathbf\Theta_+$.

	\item \textit{The hub relation}. For every word $u$ let $K(u)$ denotes the following word:
	\begin{center}
		$K(u) \equiv (u^{-1} \kappa_1 u \kappa_2 u^{-1} \kappa_3 u \kappa_4 \dots u^{-1} \kappa_{2N-1} u \kappa_{2N})\times $\\ 
		$(\kappa_{2N} u^{-1} \kappa_{2N-1}u \dots \kappa_2u^{-1}\kappa_1 u)^{-1}$.
	\end{center}
	Then the hub relation is $K(W_0)=1$, where $W_0$ is the accepting configuration of the $\mathcal{S}$-machine.
\end{enumerate}
The objective  of Sapir, Birget and Rips \cite{Sap02} in constructing such groups is to prove the following theorem :
\begin{Th} {\cite{Sap02}} {\label{supertheorem}}
Let $L \subseteq X^{+}$ be a language accepted by a Turing machine $M$ with a time function $T(n)$ for which $T(n)^4$ is superadditive.
Then there exists a finitely presented group $G(M)=\langle A \rangle$ with Dehn's function equivalent to $T(n)^4$, the smallest, isodiametric function
equivalent to $T^3(n)$, and there exists an injective map $H: X^+ \to (A \cup A^{-1})^+$ such that 
\begin{enumerate}
	\item $u \in L$ if and only if $H(u)=1$ in $G$;
	\item $H(u)$ has length $O(|u|)^2$ and is computable in time $O(|u|)$.
\end{enumerate}
\end{Th}

\section{The machine $\mathcal{S}_4$}

As we already saw it, the construction of $\mathcal{S}$-machine in \cite{Sap02} involves eleven others $\mathcal{S}$-machine.
We shall focus on four machines namely $\mathcal{S}_1$,$\mathcal{S}_2,\mathcal{S}_3,\mathcal{S}_4$ of \cite{Sap02}. As we will see the combination of 
some rule from these machines
allows to construct words that deny the necessary condition the following statement :

\begin{Stat}\label{Grom} Suppose that an asymptotic cone $Con_\omega(G,d)$ is simply connected then for every $M > 1$ there exists a number 
$k$ such that for every constant $C \geq 1$, every loop $l$ in the Cayley graph of $G$ satisfying 
$\frac{1}{C}d_m \leq |l| \leq Cd_m$ for any sufficient large $m$, bounds a disc that can be subdivided into $k$ subdisc with perimeter
at most $\frac{|l|}{M}$.
\end{Stat}
The reader will find a reference of this statement in \cite{Sap05}.
Therefore such words will ensure that the asymptotic cones of the 
group $G_N(\mathcal{S})$ are not simply connected. The result is independent of the Turing machine considered and thus can be concluded 
for each group $G_N(\mathcal{S})$ constructed following \cite{Sap02}. Once the words are constructed, the proof works roughly as the one in \cite{Sap05}. 
First we need to explain how is constructed
the machine $\mathcal{S}_4$, this is a critical step in the proof. Formally the machine $\mathcal{S}_4$ is constructed from 
$\mathcal{S}_1,\mathcal{S}_2,\mathcal{S}_3$.

Let us describe the machine $\mathcal{S}_1$. Its hardware is :
\begin{itemize}
	\item $Y(1)=(\{\delta\},\{\delta\},\{\delta\},\{\delta\},\{\delta\})$
	\item $Q(1)=(\{p_1,p_2,p_3\},\{q_1,q_2,q_3\},\{r_1,r_2,r_3\},\{s_1,s_2,s_3\},\{t_1,t_2,t_3\},\{u_1,u_2,u_3\})$.
\end{itemize}
The admissible words of $\mathcal{S}_1$ have the following form:
\begin{equation}
	p \delta^{n_1} q \delta^{n_2} r \delta^{n_3} s \delta^{n_4} t \delta^{n_5} u
\nonumber
\end{equation}
where $p,q,r,s,t,u$ may have indices $1,2,3$ and $n_i \in \mathbb{Z}, i \in \{1,\dots,5\}$.
The program $P(1)$ of $\mathcal{S}_1$ is constructed from the following rules and their inverses.

\begin{enumerate}
	\item $[q_1 \to \delta^{-2} q_1 \delta^2, r_1 \to \delta^{-1} r_1 \delta ]$
	\item $[p_1q_1 \to p_2 q_2, r_1 \to r_2, s_1 \to s_2, t_1 \to t_2, u_1 \to \delta u_2]$
	\item $[p_1\delta q_1 \to p_3 \delta q_3, r_1 \to r_3, s_1 \to s_3, t_1 \to t_3, u_1 \to u_3]$.
\end{enumerate}
The hardware of $\mathcal{S}_2$ is 
\begin{itemize}
	\item $Y(2)=Y(1)$,
	\item $Q(2)=(\{p_1,p_2\},\{q_1,q_2\},\{r_1,r_2\},\{s_1,s_2\},\{t_1,t_2\},\{u_1,u_2\})$.
\end{itemize}
The program $P(2)$ of $\mathcal{S}_2$ consists of the following rules and their inverses:
\begin{enumerate}
	\item $[q_2 \to \delta q_2 \delta^{-1},s_2 \to \delta^{-1} s_2 \delta ]$
	\item $[p_2 \to p_1, q_2r_2s_2 \to q_1 r_1 s_1, t_2 \to t_1, u_2 \to u_1]$.
\end{enumerate}
The machine $\mathcal{S}_3$ in \cite{Sap02} is defined as a \emph{cycle} of machines $\mathcal{S}_1,\mathcal{S}_2$. Roughly speaking 
$\mathcal{S}_3$ is obtained by taking the union of $\mathcal{S}_1$ and $\mathcal{S}_2$ and identifying two state vectors of $\mathcal{S}_1$ with two
state vector of $\mathcal{S}_2$.
The hardware of $\mathcal{S}_3$ is $(Y(3),Q(3))$ is the same as the hardware of $\mathcal{S}_1$.
The program $P(3)$ is constructed from the following rules and their inverses.
\begin{enumerate}
	\item $[q_1 \to \delta^{-2} q_1 \delta^2, r_1 \to \delta^{-1} r_1 \delta]$
	\item $[p_1 q_1 \to p_2q_2, r_1 \to r_2, s_1 \to s_2, t_1 \to t_2, u_1 \to \delta u_2]$
	\item $[p_1 \delta q_1 \to p_3 \delta q_3, r_1 \to r_3, s_1 \to s_3, t_1 \to t_3, u_1 \to u_3]$.

	\item $[q_2 \to \delta q_2 \delta^{-1}, s_2 \to \delta^{-1} s_2 \delta]$
	\item $[q_2 \to p_1, q_2 r_2 s_2 \to q_1 r_1 s_1, t_2 \to t_1, u_2 \to u_1 ]$.
\end{enumerate}

In \cite{Sap02} the machine $\mathcal{S}_4$ is constructed as a \emph{concatenation} of two copies of $\mathcal{S}_3$ with common states 
$p_3,q_3,r_3,s_3$ and $t_3$. Let $\mathcal{S}'_3$ be a copy of the machine $\mathcal{S}_3$. $\mathcal{S}'_3$ is obtained by adding $'$ to all states letters 
of $Q_3$ except $p_3,q_3,r_3,s_3,t_3$ and $u_3$. The set of states of $\mathcal{S}'_3$ is 
$Q'(3)=\{p'_1,p'_2,p_3\} \cup \{q'_1,q'_2,q_3\} \cup \{r'_1,r'_2,r_3\} \cup \{s'_1,s'_2,s_3\} \cup \{t'_1,t'_2,t_3\} \cup \{u'_1,u'_2,u_3\}$.
Let $Q(4)$ constructed as the union of $Q(3)$ and $Q'(3)$. Define $P(4)$ as the union of programs $P(3)$ and $P'(3)$ of the machines
$\mathcal{S}_3$ and $\mathcal{S}'_3$. Denote by $\mathcal{S}_4$ the machine with hardware $(Y(1),Q(4))$ and program $P(4)$.
According to Lemma 4.6 of \cite{Sap02} the machine $\mathcal{S}_4$ tells zero from nonzero and returns all state letters to their original positions.
To understand how we will construct the words that deny the statement \ref{Grom} it is useful to see the critical steps of the simulation.
In \cite{Sap02} the simulation of a Turing machine works as follows. Let $M$ be a Turing machine and $\tau$ a command of $M$ of the form 
\begin{equation}
	\tau=\{F_{q_1} \to F_{q'_1}, \dots, a F_{q_i} \to F_{q'_i},\dots,F_{q_k} \to F_{q'_k} \}.
\nonumber
\end{equation}
Remember that 
$\mathcal{S}_\gamma(\tau)$ is a copy of machine $\mathcal{S}_\gamma$. The machine $\mathcal{S}(M)$ simulates the command $\tau$ as follows.
First, using $\mathcal{S}_4(\tau)$, it is checked whether the word between $E'(i)$ and $F'_{q_i}(i)$ is empty. If it is empty, the execution cannot proceed to 
the next step. Otherwise the machine changes $q_i$ to $q'_j$ in the indices of the $F's$, inserts $a^{-1}$ next to the left of $x(i)$, removes one $\delta$ in the 
word between $E'(i)$ and $F'_{q_i}(i)$, removes one $\delta$ and removes one $\omega$. Using $\mathcal{S}_9(\tau)$ it finally checks if after $a^{-1}$,
the word between $E(i)$ and $F_{q'_j}(i)$ is positive. If it is the case the machine gets ready to execute the next transition.

The critical step for our work is the first one when the machine $\mathcal{S}_4(\tau)$ is used. Indeed it means that for each command 
$\tau$ of the machine $M$ there exists a copy of the rules of $\mathcal{S}_4$ in $\mathcal{S}(M)$. The next section shall explain 
what are the consequences in the group $G_N(\mathcal{S})$.

\section{Consequences in $G_N(\mathcal{S})$}

As we saw previously the machine $\mathcal{S}(M)$ contains copies of rules of the machine $\mathcal{S}_4(\tau)$ where $\tau$ is a command of $M$.
Remember that $G_N(\mathcal{S})$ contains in its presentation the transitions relations and that they correspond to the element of $\mathbf{\Theta_+}$.
Let $\tau \in \mathbf{\Theta_+}, \tau=[U_1 \to V_1,\dots,U_p \to V_p]$. Then relations $\tau^{-1}U_1\tau=V_1,\dots,\tau^{-1}U_p\tau=V_p$ are included 
into $P_N(\mathcal{S})$. If for some $j$ from $1$ to $17k+6$ the letters from $Q_j$ do not appear in any of the $U_i$ then the relations
$\tau^{-1} q_j \tau =q_j$ for every $q_j \in Q_j$ are also included.
Denote $\sigma_i(4,\tau)$ the copy of the rule $\sigma_i$ of $\mathcal{S}_4$. Let $\sigma_1,\sigma_4$ be the following rule of $\mathcal{S}_4$:
\begin{itemize}
	\item $\sigma_1 : [q_1 \to \delta^{-2} q_1 \delta^2, r_1 \to \delta^{-1} r_1 \delta]$,
	\item $\sigma_4 : [q_2 \to \delta q_2 \delta^{-1}, s_2 \to \delta^{-1} s_2 \delta]$.
\end{itemize}
It means that in $G_N(\mathcal{S})$ the following relation exists:
\begin{itemize}
	\item $\sigma_1(4,\tau)^{-1} q_1 \sigma_1(4,\tau) = \delta^{-2} q_1 \delta^{2}$,  
	\item $\sigma_1(4,\tau)^{-1} r_1 \sigma_1(4,\tau) = \delta^{-1} r_1 \delta$,
	\item $\sigma_4(4,\tau)^{-1} q_2 \sigma_4(4,\tau) = \delta q_2 \delta^{-1}$,
	\item $\sigma_4(4,\tau)^{-1} s_2 \sigma_4(4,\tau) = \delta^{-1} s_2 \delta$.
\end{itemize}
Moreover the following relations are also included in the presentation of $G_N(\mathcal{S})$:
\begin{itemize}
	\item $\sigma_1(4,\tau)^{-1} s_2 \sigma_1(4,\tau) = s_2$,
	\item $\sigma_4(4,\tau)^{-1} r_1 \sigma_4(4,\tau) = r_1$.
\end{itemize}
From now on and for the sake of simplicity we denote the rule $\sigma_1(4,\tau)$ (resp. $\sigma_4(4,\tau)$) by $\sigma_1$ (resp. $\sigma_4$).
Let us study briefly the word $\sigma_1^{-n} \sigma_4^{-n} (s_2r_1)^n \sigma_4^{n} \sigma_1^{n}$.

\begin{Lem}\label{cool_word} In $G_N(\mathcal{S})$ the word $\sigma_1^{-n} \sigma_4^{-n} (s_2r_1)^n \sigma_4^{n} \sigma_1^{n}$ is equal 
to $\delta^{-n} (s_2 r_1 )^n \delta^n$.
\end{Lem}
\begin{proof} First we prove by induction that 
\begin{equation}
\sigma_4^{-n} (s_2 r_1 )^n \sigma_4^{n} 
\nonumber
\end{equation}
is equal to 
\begin{equation}
	(\delta^{-n} s_2 \delta^n r_1 )^n
\nonumber
\end{equation}
If $k=0$ the equality is clear. Let $n \in \mathbb{N}$ and assume 
$\sigma_4^{-k} (s_2 r_1)^k \sigma_4^k = (\delta^{-k}s_2 \delta^k r_1)^k$ is true for $k\leq n$. We shall show that 
\begin{equation} \label{firstrel}
	\sigma_4^{-(n+1)} (s_2 r_1)^{n+1} \sigma_4^{n+1} = (\delta^{-(n+1)} s_2 \delta^{n+1} r_1)^{n+1}
\end{equation}
inserting accordingly the word $\sigma_4^{n} \sigma_4^{-n}$ we deduce from \eqref{firstrel}  
\begin{equation} \label{cooleq}
        \sigma_4^{-1} \boxed{ \sigma_4^{-n}(s_2 r_1)^n \sigma_4^{n}} \sigma_4^{-n} s_2 \sigma_4^{n} \sigma_4^{-n} r_1 \sigma_4^{n+1} 
\nonumber
\end{equation}
applying the induction hypothesis on the word in the box we obtain
\begin{equation}\label{hypded}
	\sigma_4^{-1} (\delta^{-n} s_2 \delta^n r_1)^n \sigma_4^{-n} s_2 \sigma_4^{n} \sigma_4^{-n} r_1 \sigma_4^{n+1}
\end{equation}
since $\delta \in Y_i$ for some $i \leq k$ then $\sigma_4 \delta = \delta \sigma_4$ is an auxiliary relation, then it comes from \eqref{hypded}
\begin{equation}
	\sigma_4^{-1} (\delta^{-n} s_2 \delta^n r_1)^n \sigma_4^{-n} s_2 \sigma_4^{n} r_1 \sigma_4
\nonumber	
\end{equation}
combining the relations $\sigma_4^{-1} s_2 \sigma_4=\delta^{-1} s_2 \delta$ and $\sigma_4 \delta= \delta \sigma_4$ gives
\begin{equation}
	\delta^{-n} \sigma_4^{-1} s_2 \underbrace{\delta^n r_1 \delta^{-n} s_2 \delta^n \dots}_{= (\delta^n r_1 \delta^{-n}s_2)^{n-1}} 
	r_1 \delta^{-n} s_2 \delta^{n} r_1 \sigma_4
\end{equation}
then inserting $\sigma_4 \sigma_4^{-1}$ we obtain
\begin{equation}
	\delta^{-n} \sigma_4^{-1} s_2 \underbrace{\sigma_4 \sigma_4^{-1} \delta^n r_1 \sigma_4 \sigma_4^{-1} \delta^{-n} s_2 \sigma_4 \sigma_4^{-1} 
	\delta^n \dots}_{= (\sigma_4 \sigma_4^{-1} \delta^n r_1 \sigma_4 \sigma_4^{-1} \delta^{-n} s_2)^{n-1}} 
	\sigma_4 \sigma_4^{-1} \delta^n r_1 \sigma_4 \sigma_4^{-1} \delta^{-n} s_2 \sigma_4 \sigma_4^{-1} \delta^{n} r_1 \sigma_4
\end{equation}
but since the letter $r_1$ is never involved in the rule $\sigma_4$ we have $\sigma_4^{-1} r_1 \sigma_4=r_1$ and thus combining it with the auxiliary rule, it comes
\begin{equation}
	\delta^{-n} \sigma_4^{-1} s_2 \underbrace{ \sigma_4 \delta^n r_1 \delta^{-n} \sigma_4^{-1} s_2 \sigma_4  
	\delta^n \dots}_{= (\sigma_4 \delta^n \sigma_4^{-1} r_1 \sigma_4 \delta^{-n} \sigma_4^{-1} s_2)^{n-1}} 
	\sigma_4 \delta^n \sigma_4^{-1} r_1 \sigma_4 \delta^{-n} \sigma_4^{-1} s_2 \sigma_4 \delta^{n} r_1 
\end{equation}
now we can apply relation $\sigma_4^{-1} s_2 \sigma_4= \delta^{-1} s_2 \delta$ and obtain
\begin{equation}
	(\delta^{-(n+1)} s_2 \delta^{n+1} r_1)^n (\delta^{-(n+1)} s_2 \delta^{n+1} r_1)
\end{equation}
and thus 
\begin{equation}\label{intereq}
	\sigma_4^{-(n+1)} (s_2 r_1)^{n+1} \sigma_4^{n+1} = (\delta^{-(n+1)} s_2 \delta^{n+1} r_1)^{n+1}
\end{equation}
Now we shall start from the second member of equation \eqref{intereq} and show the following
\begin{equation}\label{lasteq}
	\sigma_1^{-n} (\delta^{-n}s_2 \delta^{n} r_1)^n \sigma_1^{n}= \delta^{-n} (s_2 r_1)^n \delta^n.
\end{equation}
Inserting the word $\sigma_1^n \sigma_1^{-n}$ accordingly and using auxiliary relation we obtain
\begin{equation}
	\sigma_1^{-n} \delta^{-n} s_2 (\sigma_1^n \delta^n \sigma_1^{-n} r_1 \sigma_1^n \delta^{-n} \sigma_1^{-n} s_2)^{n-1} 
	\sigma_1^n \delta^n \sigma_1^{-n} r_1 \sigma_1^n
\end{equation}
using the relations $\sigma_1^{-1} s_2 \sigma=s_2$ and $\sigma_1 \delta=\delta \sigma_1$ leads to 
\begin{equation}
	\delta^{-n} s_2 (\delta^n \sigma_1^{-n} r_1 \sigma_1^{n} \delta^{-n} s_2)^{n-1} \delta^n \sigma_1^{-n} r_1 \sigma_1^n
\end{equation}
applying $\sigma_1^{-1} r_1 \sigma_1= \delta^{-1} r_1 \delta$ and the auxiliary relation $\sigma_1 \delta = \delta \sigma_1$ gives 
\begin{equation}
	\delta^{-n} (s_2 r_1)^n \delta^n.
\end{equation}
therefore the equality is proved.
\end{proof}

Lemma \ref{cool_word} allows one to consider van Kampen diagram $\Delta_n$ with a boundary labeled by the word 
\begin{equation}
\sigma_1^{-n} \sigma_4^{-n}(s_2 r_2)^n \sigma_4^n \sigma_1^n \delta^{-n} (s_2 r_1)^n \delta^n.
\end{equation}
We shall use such van Kampen diagram to deny the necessary condition of \textbf{Statement \ref{Grom}}. 
That is we show that no loop corresponding to $\Delta_n$ can bound a disc decomposed into at most $k \leq \sqrt{n}$.
First let us recall what is an $x$-band, for every letter $x$. An $x$-edge in a van Kampen diagram is an edge labeled by $x^{\pm 1}$.
An $x$-cell is a cell whose boundary contains an $x$-edge. An $x$-band in a diagram is a sequence of cells containing $x$-edges,
such that every two consecutive cells share an $x$-edge. The boundary of the union of cells from an $x$-band $\mathcal{B}$
has the form $s^{-1}peq^{-1}$ where $s,e$ are the only $x$-edges on the boundary representing respectively the start and the end of the 
band. The paths $p,q$ are called the \emph{sides} of $\mathcal{B}$.

\begin{Th} Let $\omega$ be a non-principal ultrafilter and $(d_i), i \in \mathbb{N}$ an increasing sequence of numbers with 
$\text{lim } d_i = \infty$. Let $Con_\omega(G_N(\mathcal{S}),(d_i))$ be an asymptotic cone of $G_N(\mathcal{S})$. Then 
$Con_\omega(G_N(\mathcal{S}),(d_i))$ is not simply connected.
\end{Th}
\begin{proof}
Let $Con_\omega(G_N(\mathcal{S}),(d_i))$ be an asymptotic cone of $G_N(\mathcal{S})$. Fix $n=d_m$ for a large $m$.
Let $u_n=\sigma_1^{-n} \sigma_4^{-n}(s_2 r_1)^n \sigma_4^n \sigma_1^n \delta^{-n} (s_2 r_1)^n \delta^n$
and $\Delta_n$ the corresponding van Kampen diagram. The top path $t$ of $\Delta_n$ is labeled by $(s_2 r_1)^n$, the bottom path $b$
is labeled by $\delta^{-n} (s_2 r_1)^n \delta^n$. The left and right sides, $l,r$ are labeled by $\sigma_1^{n} \sigma_4^n$.
The perimeter of $\Delta_n$ is $|\Delta_n|=8n$. We shall show that a loop in the Cayley graph of $G_N(\mathcal{S})$ corresponding to
$u_n$ cannot bound a disc decomposed into at most $k \leq \sqrt{n}$ subdiscs of perimeter $n$. Assume that such a decomposition exists, there is
a van Kampen diagram $\Delta$ with boundary label $u_n$ composed of $k$ subdiagrams $\Delta_1,\dots,\Delta_k$ with perimeter at most $n$.
Consider any $\sigma_4$-edges $e$ of the path $l$. Since there is no $\sigma_4$-edges on path $t,b$ any $\sigma_4$-band in $\Delta$ that starts at $e$ 
cannot end on $t,b$. Therefore it finishes on $r$. Moreover the $\sigma_4$-bands do not intersect and thus they connected corresponding $\sigma_4$-edges.
Let $e$ be the $\sigma_4$-edges number $n$ on $l$ and $\mathcal{B}$ be the maximal $\sigma_4$-band starting at $e$. Let $r$ be the top side
of $\mathcal{B}$. The label $Lab(r)$ of path $r$ belongs to the free group $\langle \delta,s_2,r_1 \rangle$ and it is equal to 
$\sigma_4^{-n} (s_2 r_1)^n \sigma^n$ in $G_N(\mathcal{S})$, thus it can be written $(\delta^{-n} s_2 \delta^n r_1)^n$.
Since the number of diagrams $\Delta_i$ is less than $\sqrt{n}$ there is a subpath $w$ of $r$ such that the initial and the terminal vertices of $w$
belong to the boundary of one of the $\Delta_i$. $w$ contains the subword $s_2 \delta^n r_1$. It means that there exists in $G_N(\mathcal{S})$ a word 
$u$ such that $u=w$ and $|u|\leq \frac{n}{2}$ since the perimeter of $\Delta_i$ does not exceed $n$.
Therefore we can consider a reduced diagram $\Gamma$ with boundary $p_1 p_2^{-1}$ where $Lab(p_1)=w,Lab(p_2)=u$. 
We look at the subword $t_1 = s_2 \delta^n r_1$. Let $\mathcal{B}_1$ the maximal $s_2$-band starting on $t_1$ and $\mathcal{B}_2$
the maximal $r_1$-band starting on $t_1$. Denote by $\Gamma_1$ the subdiagram of $\Gamma$ bounded by $q_1=t_1 \setminus \{s_2,r_1\}$, the sides of 
$\mathcal{B}_1$ and $\mathcal{B}_2$ and a part $q_2$ of $p_2$. Let $\partial \Gamma_1=q_1 q'_2$ the boundary of $\Gamma_1$.
We shall bound the length of $q'_2$. One can remark that a cell appearing in a $s_2$-band in $\Gamma$ can be written in the form 
$s_2=\delta \sigma_4^{-1} s_2 \sigma_4 \delta^{-1}$. That is the length of a side of $\mathcal{B}_1$ is at most twice the number of $\sigma_4$-edges in it.
The number of $\sigma_4$-edges on $\mathcal{B}_1$ and $\mathcal{B}_2$ is at most $|p_2| - |q_2| - 2$. Thus we have $|q'_2| \leq 2|p_2| -2|q_2| +q_2 - 4$ and then 
it comes $|q'_2| < 2|p_2| - |q_2|$. It is not difficult to see that the diagram $\Gamma_1$ can be chosen such that its top path $q_2=s_2 \delta^k r_1$.
Moreover in a such diagram there are no $s_2$-cells and thus every $\delta$-band starting on $q_1$ must end on $q'_2$. But this is a contradiction since 
$q'_2 < n$. 

\begin{Rem} The proof works roughly as the proof in  \cite{Sap05}.
\end{Rem}
\end{proof}
\bibliography{biblio}
\bibliographystyle{unsrt}
\end{document}